\documentclass[12pt]{amsart}
\usepackage{geometry}                % See geometry.pdf to learn the layout options. There are lots.
\usepackage{amssymb,latexsym,amsfonts, amsthm, amsmath, amsrefs,mathrsfs}%, showkeys}
\usepackage{enumerate}
\usepackage{fancyhdr}
\usepackage{epigraph}
\usepackage{verbatim}
\usepackage{array}
\usepackage{blkarray}
\usepackage{url}
\usepackage{fullpage}
\usepackage[vcentermath, enableskew]{youngtab}
\usepackage{multicol}
\usepackage{longtable}
\usepackage[font=small,labelfont=bf,singlelinecheck=on]{caption}

\usepackage{tikz}
\usetikzlibrary{matrix, shapes, arrows, calc, backgrounds}
\usetikzlibrary{arrows,chains,matrix,positioning,scopes}
\usepackage{pgfplots}

\makeatletter
\tikzset{join/.code=\tikzset{after node path={%
\ifx\tikzchainprevious\pgfutil@empty\else(\tikzchainprevious)%
edge[every join]#1(\tikzchaincurrent)\fi}}}
\tikzset{>=stealth',every on chain/.append style={join},
         every join/.style={->}}
          
\usepackage{color} 

\newtheorem{thm}{Theorem}[section]   % Numbered within each section
\newtheorem{prop}[thm]{Proposition}  % Numbered along with thm

\theoremstyle{definition} 
\newtheorem{defn}[thm]{Definition}   % Numbered along with thm
\newtheorem{ex}[thm]{Example}        % Numbered along with thm
\newtheorem{rmk}[thm]{Remark}

\newtheorem{notation}[thm]{Notation}

\newcommand{\jac}{\mathcal{J}}
\newcommand{\NM}{\mathcal{N}\hspace{-1mm}\mathcal{M}}

\newcommand{\height}{\operatorname{ht}}

\newcommand\independent{\protect\mathpalette{\protect\independenT}{\perp}}
\def\independenT#1#2{\mathrel{\rlap{$#1#2$}\mkern5mu{#1#2}}}

\newcommand{\mc}{\ensuremath{\mathcal}}
\newcommand{\ol}{\ensuremath{\overline}}

\newcommand{\ZZ}{\ensuremath{\mathbb{Z}}}
\newcommand{\RR}{\ensuremath{\mathbb{R}}}
\newcommand{\PP}{\ensuremath{\mathbb{P}}}
\newcommand{\QQ}{\ensuremath{\mathbb{Q}}}
\newcommand{\FF}{\ensuremath{\mathbb{F}}}
\newcommand{\CC}{\ensuremath{\mathbb{C}}}
\newcommand{\NN}{\ensuremath{\mathbb{N}}}

\title{Computing Algebraic Matroids}
\author{Zvi Rosen}

\address{Department of Mathematics, University of California, Berkeley, CA 94720}
\email{zhrosen@math.berkeley.edu}

\date{\today}                                           % Activate to display a given date or no date

\begin{document}
\maketitle

\begin{abstract}
An affine variety induces the structure of an \emph{algebraic matroid} on the set of coordinates of the ambient space. The matroid has two natural decorations: a {\it circuit polynomial} attached to each circuit, and the degree of the projection map to each base, called the {\it base degree}. 
Decorated algebraic matroids can be computed via symbolic computation using Gr\"{o}bner bases, or through linear algebra in the space of differentials (with decorations calculated using numerical algebraic geometry). Both algorithms are developed here. Failure of the second algorithm occurs on a subvariety called the \emph{non-matroidal} or \emph{NM-locus}. Decorated algebraic matroids have widespread relevance anywhere that coordinates have combinatorial significance. Examples are computed from applied algebra, in algebraic statistics and chemical reaction network theory, as well as more theoretical examples from algebraic geometry and matroid theory.
\end{abstract}

\section{Introduction}

Algebraic matroids have a surprisingly long history. They were studied as early as the 1930's by van der Waerden, in his textbook Moderne Algebra \cite[Chapter VIII]{vdw}, and MacLane in one of his earliest papers on matroids \cite{maclane}. The topic lay dormant until the 70's and 80's, when a number of papers about representability of matroids as algebraic matroids were published: notably by Ingleton and Main \cite{IM75}, Dress and Lovasz \cite{DL87}, the thesis of Piff \cite{Piff}, and extensively by Lindstr\"{o}m (\cite{Lind83,Lind86,Lind87,Lind88}, among others). In recent years, the algebraic matroids of \emph{toric} varieties found application (e.g. in \cite{Hibi}); however, they have been primarily confined to that setting. 

Renewed interest in algebraic matroids comes from the field of matrix completion, starting with \cite{KTTU}, where the set of entries of a low-rank matrix are the ground set of the algebraic matroid associated to the determinantal variety.  In applied algebra in general, coordinates typically carry real-world significance, and the matroid has inherent interest as the dependence structure among those quantities.  Even for varieties arising in pure mathematics, distinguished coordinates may have combinatorial meaning, in which case the matroid also provides insight. From this perspective, algebraic matroids deserve serious study, which will be aided by computational tools.

Matroids have many characterizations; the computation problem tackled in this paper is to take as input a prime ideal in a polynomial ring with specified generators, and return as output the list of bases and circuits of the corresponding matroid.  Two algorithms are presented in Section \ref{alg}: one uses symbolic computation in the polynomial ring with Gr\"{o}bner bases, while the other passes to an equivalent linear matroid, where computations can use simple linear algebra (This technique was implemented in \cite{KTTU}). Numerical algebraic geometry \cite{bertiniBook} also plays a helpful role in attaching algebraic ``decorations" to the matroid.

We use the following notation, in keeping with the standard texts \cite{Ox} and \cite{Welsh}:

\begin{notation} $\mc{M}$ denotes the matroid, $E$ the \emph{ground set} of the matroid, $\mc{I}(\mc{M})$ or $\mc{I}$ the set of \emph{independent sets}, $\mc{B}(\mc{M})$ or $\mc{B}$ the set of \emph{bases}, $\mc{C}(\mc{M})$ or $\mc{C}$ the set of \emph{circuits}, and given $S \subset E$, $\rho(S)$ denotes the value of the \emph{matroid rank} function on $S$. We will say that two matroids are \emph{identical}, if the obvious map of ground sets induces bijection on $\mc{I},\mc{B},\mc{C},\rho$, etc.

\begin{comment}
\centerline{\begin{tabular}{ll}
$\mc{M}$ & Matroid. \\
$E$ & Ground set of the matroid. \\
$\mc{B}(\mc{M})$ & Bases of the matroid. \\
$\mc{C}(\mc{M})$ & Circuits of the matroid. \\
$\rho(S) : S \subset E$ & Rank of the matroid.
\end{tabular}}\end{comment}
\end{notation}

\begin{defn}

Let $k$ be a field, and $R =k[x_1,\ldots,x_n]$, a polynomial ring. Let $P \subseteq k[x_1,\ldots,x_n]$ be a prime ideal. The ring $S = k[x_1,\ldots,x_n]/ P $ is an integral domain, so the function field $K = $ Frac$(S)$ is well-defined.  Let $E = \{\ol{x_1},\ldots,\ol{x_n}\} \subset K$ be the image of the variables under the composition of the quotient and the injection $\varphi: k \to S \to K$. Independence is defined as usual in an algebraic matroid: algebraic independence over the ground field $k$. $\mc{M}(P)$ denotes the matroid obtained from a prime ideal in this manner.
\end{defn}

In fact, every algebraic matroid $\mc{M}$ can be obtained as $\mc{M}(P)$ for some prime ideal $P$. Start from an algebraic matroid $\mc{M}$ of size $n$ with ground set contained in $K/k$. Set the ground set $E$ to be the image of the variables in a ring map $\phi_E: k[x_1,\ldots,x_n] \to K$. The image of the induced map of varieties is an irreducible variety in $k^n$. The associated prime ideal $P$, obtainable by implicitization satisfies $\mc{M} = \mc{M}(P)$.  Despite this property, it is often convenient to study the matroid of a variety purely in terms of the variety's parametrization. Given a map $\phi: k[x_1, \ldots,x_n] \to K$, the notation $\mc{M}(\phi)$ will then refer to the algebraic matroid with ground set $\{ \phi(x_i) : i =1,\ldots,n \}$.

\subsection{Decorated Bases and Circuits}
\label{subsec:dbc}

We can infuse more of the algebraic structure of the ideal into the matroid via ``matroid decorations". This approach of enhancing a matroid with more information has been taken in various forms: oriented matroids \cite{orient}, arithmetic matroids \cite{ari}, valuated matroids \cite{val} and matroids over rings \cite{matring}, to name a few. Circuits of algebraic matroids have a natural decoration, based on the following fact (\cite[Lemma 5.6]{KRT13}):

\vspace{2mm}

\begin{tabular}{p{.25\textwidth}cp{.6\textwidth}}
$C = \{x_{i_1},\ldots,x_{i_k}\}$ is a 

circuit of $\mc{M}(P)$ & $\iff$ & The ideal $I \cap k[C]$ is principal with generator $\theta_C$ 

s.t. support$(\theta_C) = C.$
\end{tabular} 
\\

The generator $\theta_C$ of the principal ideal, called a \emph{circuit polynomial} is unique up to unit. This invariant was used by Dress and Lovasz in \cite{DL87} as well as Lindstr\"{o}m in \cite{Lind87} in the process of proving structural facts about algebraic matroids. More recently in \cite{KRT13}, the circuit polynomials are studied as objects of interest in their own right.  If the polynomial itself is too unwieldy, we may prefer to record only some aspect of the polynomial: (1) The {\bf Newton Polytope}  associated to the polynomial.
(2) {\bf Top-Degree}: This aspect is explored in \cite{KRT13}. It is the vector in $\NN^n$ given by $(\deg_{x_i} \theta_C)_{i \in [n]}$. Equivalently, it is the outer vertex of the tightest bounding box for the Newton polytope. When we try to construct points in a variety based on subsets of the coordinates, the top-degree allows us to determine the cardinality of the solution set.
(3) {\bf Degree}: A natural concise invariant.

\begin{rmk} An open question related to the degrees is as follows:
\emph{What constraints are there for a set of integers attached to the circuits of a matroid to be the degrees of circuit polynomials for some algebraic matroid?} If $\mc{M}$ is a matroid with corank $1$, then there is a unique circuit: the full set of variables. Assuming there is more than one variable, we can find an irreducible circuit polynomial of any degree involving the circuit variables.  On the other hand, suppose that a matroid has loops. Over an algebraically closed field, the degree of a loop's polynomial is forced to $1$, since a polynomial in one variable breaks into linear factors and the polynomial must be irreducible. Over $\QQ$, on the other hand, any degree is possible for a loop, since there are irreducible polynomials in every degree. (See Section \ref{sec:matrep} for another example where the field places constraints on the matroid decorations.) We plan to explore this question further in later research.
\end{rmk}

\noindent The bases of an algebraic matroid have the following nice property (cf. \cite[Definition 2.6]{KRT13}):  

\vspace{1mm}

\hspace{-5mm}\begin{tabular}{rcp{.6\textwidth}}
$\{x_{i_1},\ldots,x_{i_k}\}$ is a base of $\mc{M}(P)$ & $\iff$ &  the projection from the variety onto the $i_1,\ldots,i_k$-th  

coordinates is surjective with finite fibers.
\end{tabular} 

The cardinality of the generic fiber (or, equivalently, the degree of the projection) is the decoration on the bases and will be referred to as the \emph{base degree}.

\section{Algorithms and Software}
\label{alg}

We will outline two strategies for computing decorated algebraic matroids: the \emph{symbolic algorithm} in Section \ref{subsec:symbolic} and the \emph{linear algorithm} in Section \ref{subsec:linear}. In both approaches, an ideal in a polynomial ring is taken, and a list of bases with base degrees and a list of circuits with circuit polynomial (or alternative decoration) are returned. Within each of these regimes, we employ techniques on two different levels: \emph{oracles}, which extract matroid features from the ideal, and {\it matroid algorithms}, which turn one type of matroid data (e.g. rank) into another type (e.g. circuits).  

Matroid algorithms are well-studied, and will not be the focus of the paper, though our software does rely upon them. In most cases, we use na\"{i}ve matroid algorithms, though we have also implemented sophisticated methods like ReverseSearch \cite{reverse} and the circuit enumeration algorithm of Boros {\it et al.} \cite{begk03} to list bases and circuits respectively. These occasionally perform better than the na\"{i}ve algorithms; however, in the majority of cases, they do not accelerate the computation.

\subsection{Symbolic Algorithm} \label{subsec:symbolic} In the symbolic realm, elimination is at the core of the computations. It is used in the \emph{rank oracle}, which will be justified by the following proposition:
\begin{prop}[Rank Oracle]
\label{rank}
 Let $\mc{M}(P)$ be a matroid as above, and $S \subset E$.
\[ \rho(S) = |S| - \height(P \cap k[S]). \]
where $\height$ denotes the height of the ideal. 
\end{prop} 
\noindent Based on this, elimination of $E\setminus S$ followed by computation of height will serve as a rank oracle. Matroid algorithms use the rank oracle to enumerate bases and circuits.  For circuits, we may also use the characterization of circuits in Section \ref{subsec:dbc} to define a \emph{circuit oracle}: test a set $S$ by (1) computing the elimination ideal $P \cap k[S]$, and (2) checking that the generator has full support on the variables of $S$. 

The decoration of the circuits is a natural byproduct of the symbolic algorithm. For the base degree, fix a base $B  = \{x_{i_1},\ldots,x_{i_d}\}$.  Under the projection map from $k^n \to k^d$ onto the coordinates in the base, the preimage of a generic point $(\lambda_1,\ldots,\lambda_d)$, generated randomly, is a subspace with defining ideal $I_B = \langle x_{i_s} - \lambda_s \: | \: s= 1,\ldots,d \rangle$. The fiber of the projection then has defining ideal $P + I_B$, and the degree of that ideal is the base degree.

We have written a file {\tt matroids.m2}, which implements the symbolic approach in the commutative algebra platform Macaulay2 \cite{M2}. This file can be found at:  \url{http://math.berkeley.edu/~zhrosen/matroids.html}. 
 
 Some important commands are listed in Table \ref{m2code}.\begin{longtable}{|r|l|} \hline
{\tt bases} & \begin{tabular}{l} \emph{Input:} Polynomial Ring, Ideal. \\ \emph{Output:} List of bases of the matroid $\mc{M}(I)$.\end{tabular} \\ \hline
{\tt decoratedBases} & \begin{tabular}{l} \emph{Input:} Polynomial Ring, Ideal. \\ \emph{Output:} List of bases of the matroid $\mc{M}(I)$ with base degree.\end{tabular} \\ \hline
{\tt circuits} & \begin{tabular}{l} \emph{Input:} Polynomial Ring, Ideal. \\ \emph{Output:} List of circuits of the matroid $\mc{M}(I)$. 
\end{tabular} \\ \hline
{\tt pCircuits} & \begin{tabular}{p{.7\textwidth}} \emph{Input:} Polynomial Ring, Ideal. \\ \emph{Output:} List of ordered pairs: circuits of the matroid $\mc{M}(I)$ together with circuit polynomials. \\ Degree-decorated circuits can be computed with:\\ {\tt apply(pCircuits, c -> (c\_0, degree(c\_1)))}. \end{tabular} \\ \hline
{\tt topDegree} & \begin{tabular}{p{.7\textwidth}} \emph{Input:} Polynomial Ring, Polynomial. \\ \emph{Output:} Top-degree vector of the polynomial w.r.t. the variables of the ring. \\ Top-degree decorated circuits can be computed with:\\ {\tt apply(pCircuits(Ring, Ideal), c -> (c\_0, topDegree(Ring,c\_1)))}. \end{tabular} \\ \hline 
\caption{Commands to compute algebraic matroids using {\tt matroids.m2}}
\label{m2code}
\end{longtable}

The code has two sources of complexity - the complexity of elimination of variables via Gr\"{o}bner bases, and the combinatorial complexity of listing and testing all potential bases, resp. circuits. For this reason, the code has difficulty with large ground sets, large-rank matroids, and ideals with high degree generators.  In trials, the code works quickly for matroids with $|E| \leq 18, \rho(\mc{M}) \leq 6$, and generators in degree $\leq 4$. For larger or higher-rank matroids, one should use a more tailored approach, as in Example \ref{kaie}.

\subsection{Linear Algebra} \label{subsec:linear}

A classical result in the study of algebraic matroids states: algebraic matroids defined over a field $k$ of characteristic zero can also be realized as a linear matroid over a field $k(T)$ where $T$ is a finite set of transcendentals (\cite[Proposition 6.7.10]{Ox}).  In particular, when $P$ is defined by generators $\langle f_1,\ldots, f_m \rangle \subseteq k[x_1,\ldots,x_n]$, we define the Jacobian matrix $\jac(P)$ as:
\begin{equation}
\left(\frac{\partial f_i}{\partial x_j}\right) \: : \: 1 \leq i \leq m, \: 1 \leq j \leq n.
\end{equation}
This matrix, when considered as a matroid with columns as the ground set and linear independence over Frac$(k[{\bf x}]/P)$ defining the independent sets $\mc{I}$, represents the {\bf dual matroid} to $\mc{M}(P)$. 
Though the derivatives are computed symbolically in the polynomial ring $k[x_1,\ldots,x_n]$, we then consider linear algebra over the function field of the variety.

When the variety is defined by a parametrization $\phi(t_1,\ldots,t_d) = (g_1({\bf t}),\ldots, g_n({\bf t}))$, we write $\jac(\phi)$ for the Jacobian matrix of the following form:
\begin{equation}
\left(\frac{\partial g_j}{\partial t_i}\right) \: : \: 1 \leq i \leq d, \: 1 \leq j \leq n.
\end{equation}
Note that the indices in top and bottom are flipped. This matrix, again setting the columns as $E$ and using linear independence over Frac$(k[{\bf x}]/P)$ to define $\mc{I}$, represents $\mc{M}(P)$ (not its dual). Since symbolic computation is more costly, certain values of $\overline{x_1},\ldots,\overline{x_n}$, the ambient coordinates, can be substituted for the variables.

\begin{defn}[NM-Locus]
Let the {\it non-matroidal locus} $\NM(I)$ denote the locus of points in $\mc{V}(I)$, at which the specialization of the Jacobian matrix does \emph{not} represent the dual of the algebraic matroid. Similarly, $\NM(\phi)$ is the set of points in the parameter space where the specialization of the Jacobian matrix does not represent the algebraic matroid.
\end{defn}

This pair of definitions specifies the values that should be avoided when specializing the linear matroid. To help describe the non-matroidal locus, we set the following notation: $I_d(M)$ will refer to the ideal generated by $d \times d$ minors of a matrix $M$. Further, $M\{S\}$ denotes the submatrix of $M$ obtained by restricting to the columns with indices in $S$.

\begin{prop} Let $V = \mc{V}(P)$ be a variety of dimension $d$ in an ambient space of dimension $n$, with Jacobian $\mc{J}(P)$ representing the dual of $\mc{M}(P)$. Then $\NM(P)$ is defined by the ideal:
\[ 	I = 	\bigcap_{B \in \mc{B}(\mc{M})} I_{n-d}(\jac(P)\{B\}),   \]
or, equivalently, the intersection of $I_{n-d}(\jac(P)\{S\})$ over all $S$ such that $I_{n-d}(\jac(P)\{S\}) \not\subseteq P$. In the special case where $\mc{J}(P)$ has $n-d$ rows, this is a principal ideal generated by the lcm of all nonzero (mod $P$) maximal minors.
\end{prop}
\begin{proof}
A matroid is fully described by its list of bases. Given any cobase of $\mathcal{M}(V)$, the corresponding $m \times (n-d)$ matrix has rank $n-d$, so some $(n-d)\times(n- d)$ minor is nonvanishing. The last fact follows from the properties of intersections of principal ideals.
\end{proof}

\begin{ex} We compute the non-matroidal locus of a torus in $\RR^3$. Let $P = \langle (x^2 + y^2 + z^2 + 3)^2 - 16(x^2 + y^2) \rangle \subseteq \RR[x,y,z]$, the defining ideal for the torus with minor radius $1$ and major radius $2$. The Jacobian $\jac(P)$ is a $1 \times 3$ matrix:
\[ \begin{array}{ccc} x & y & z \\[1mm] [4x^3+4xy^2+4xz^2-20x & 4x^2y+4y^3+4yz^2-20y & 4x^2z+4y^2z+4z^3+12z] \end{array} \]
Since the dual matroid has rank $1$, the non-matroidal locus $\NM(P)$ is a principal ideal generated by the lcm of the entries.
\[ \NM(P) = \langle -x^5yz-2x^3y^3z-xy^5z-2x^3yz^3-2xy^3z^3-xyz^5+2x^3yz+2xy^3z+2xyz^3+15xyz \rangle.\]
We add the ideal to $P$ to find the non-matroidal locus as a subvariety of the torus, and we compute the associated primes:
\[ \langle x,y^2+z^2+4y+3\rangle, \langle x,y^2+z^2-4y+3\rangle, \langle y,x^2+z^2+4x+3\rangle, \langle y,x^2+z^2-4x+3\rangle,\]
\[      \langle z,x^2+y^2-9\rangle, \langle z,x^2+y^2-1\rangle, { \bf \langle z^2+3,x^2+y^2\rangle}, \langle z+1,x^2+y^2-4\rangle, \langle z-1,x^2+y^2-4\rangle.\]
The boldface ideal has empty real variety, but the other $8$ ideals define $8$ circles around the torus, four for each generator of the fundamental group. Specializing at any point on those circles gives the wrong matroid for $\mc{M}(P)$.
\end{ex}

The corresponding statement for parametrized varieties is given in Proposition \ref{paramjac}.
\begin{prop} Let $V = \mc{V}(\phi)$ be a variety of dimension $d$ with Jacobian $\mc{J}(\phi)$ representing $\mc{M}(\phi)$. Then $\NM(\phi)$ is defined by the ideal:
\[ 	I = 	\bigcap_{B \in \mc{B}(\mc{M})} I_d(\jac(\phi)\{B\}),   \]
or, equivalently, the intersection of all nonzero $I_d(\jac(\phi)\{S\})$. This is a principal ideal generated by the lcm of all nonzero maximal minors.
\label{paramjac}
\end{prop}

For the linear algorithm, the goal is to specialize the Jacobian at a point so that we can perform linear algebra to compute the matroid. The propositions imply that selecting a non-root of the ideal $I$ guarantees the desired matroid. From the formulation, it is clear that the non-matroidal locus has positive codimension in the ambient variety; therefore, selection of a generic point is sufficient to guarantee that the NM-locus is avoided.

For parametrized varieties, the Jacobian at a generic point is obtained simply by plugging in random numbers for each parameter. For varieties defined by ideals, a generic point can be computed using numerical algebraic geometry software; we use {\tt Bertini} \cite{Bertini}. 

\begin{rmk}
When we select points in the variety numerically, we often need to use very high accuracy. A set of columns may have minors with polynomial values that evaluate to zero when passed to the quotient; however, when we specialize to a point with low accuracy, we may find that the minors corresponding to these columns are $\gg 0$. The required accuracy depends on the degree of the ideal generators or polynomial parametrizations.
\end{rmk}

Software that computes linear matroids is then used to transform the matrix into a list of circuits and bases; we use numerical linear algebra in {\tt Sage} \cite{sage}, as well as its {\tt Matroid} implementation. These lists are translated into a set of $\{0,1\}$ vectors that are sent back to Bertini, which computes coordinate projections in {\tt TrackType:5}. Bertini performs each projection with base degree for the list of basis vectors, and degree or top-degree of the circuit polynomial for the list of circuit vectors.  Bertini returns these values using numerical algebraic geometry techniques (see \cite{bertiniBook} for more details). In this mode of computation, Gr\"{o}bner basis complexity is avoided; however, combinatorial complexity is still a fixture. The original calculation of the witness set can also be expensive for a high-dimensional variety and ambient space. Examples of code for Sage and Bertini are included in the aforementioned website.

\section{Applications}

In this section, examples from different areas of mathematics will demonstrate that decorated algebraic matroids are natural and provide fundamental insight into the independence structure of a system of distinguished coordinates. As mentioned earlier this is an approach which has already been explored in matrix completion \cite{KTTU}, and which can be applied much more broadly.

\subsection{Algebraic Statistics}

The first area we look into is the field of algebraic statistics. Statistical models have distinguished coordinates describing the probability of an event; the relationship among those coordinates is therefore an obvious and natural question. The decorated algebraic matroid is the way to succinctly describe the independence structure among those probabilities.  In this section, we discuss two specific models from \cite{SW14,KRS14}.

\begin{ex}[$PL_4$ matroid] $\begin{array}{|ccccc|} \hline
\mc{M}(PL_4): & |E| = 24, & \rho = 4, & |\mc{B}| = 10560, & |\mc{C}| = 41346. \\ \hline
\end{array}$

Consider a random variable $X$ which takes as values the permutations of the letters $1 \cdots n$, which correspond to rankings of a set of preferences. Probability functions $p_{\pi}: \Theta \to [0,1]$ assign probabilities to each ranking as a function of some parameters $\theta_1,\ldots,\theta_k$.

\begin{comment}
\hspace{-4mm}
\begin{minipage}{.5\textwidth}
\vspace{5mm}
The geometry of the variety defined by the image of the $p_\pi$'s in $[0,1]^{n!}$ is the object of interest. We forget the cube and consider the variety in $\CC^n$ for simplicity.  In \cite[Section 7]{SW14}, the Plackett-Luce model is defined by:
\[ p_\pi \mapsto \prod_{i = 1}^{n-1} \frac{1}{\sum_{j = 1}^i \theta_{\pi(j)}}.\]
Since algebraic dependence is not changed by reciprocating elements, we instead consider $p_\pi \mapsto \prod_{i = 1}^{n-1}(\sum_{j = 1}^i \theta_{\pi(j)})$ for easier computation.  The variety defining the Plackett-Luce model for $n = 4$ is $4$-dimensional with degree $27$; the corresponding ideal is minimally generated by $9$ polynomials of degree $1$ and $36$ degree-$2$ polynomials.  Since $\mc{M}(PL_4)$ has rank $4$, the matroid may be represented by an affine representation in $3$-space, as in Figure \ref{pl4fig}.
\end{minipage}
\begin{minipage}{.5\textwidth}
\centering
\includegraphics[scale=.5]{plackettLuce.png}
\captionof{figure}{Affine Representation of $\mc{M}(PL_4)$}
\label{pl4fig}
\end{minipage}
\end{comment}

The geometry of the variety defined by the image of the $p_\pi$'s in $[0,1]^{n!}$ is the object of interest. We forget the cube and consider the variety in $\CC^{n!}$ for simplicity.  In \cite[Section 7]{SW14}, the Plackett-Luce model is defined by:
\[ p_\pi \mapsto \prod_{i = 1}^{n-1} \frac{1}{\sum_{j = 1}^i \theta_{\pi(j)}}.\]
Since algebraic dependence is not changed by reciprocating elements, we instead consider $p_\pi \mapsto \prod_{i = 1}^{n-1}(\sum_{j = 1}^i \theta_{\pi(j)})$ for easier computation.  The variety defining the Plackett-Luce model for $n = 4$ is $4$-dimensional with degree $27$; the corresponding ideal is minimally generated by $9$ polynomials of degree $1$ and $36$ degree-$2$ polynomials.  Since $\mc{M}(PL_4)$ has rank $4$, the matroid may be represented by an affine representation in $3$-space, depicted by its Schlegel diagram in Figure \ref{pl4fig}, made with Polymake \cite{polymake}.

\begin{figure}[!h]
\includegraphics[scale=.6]{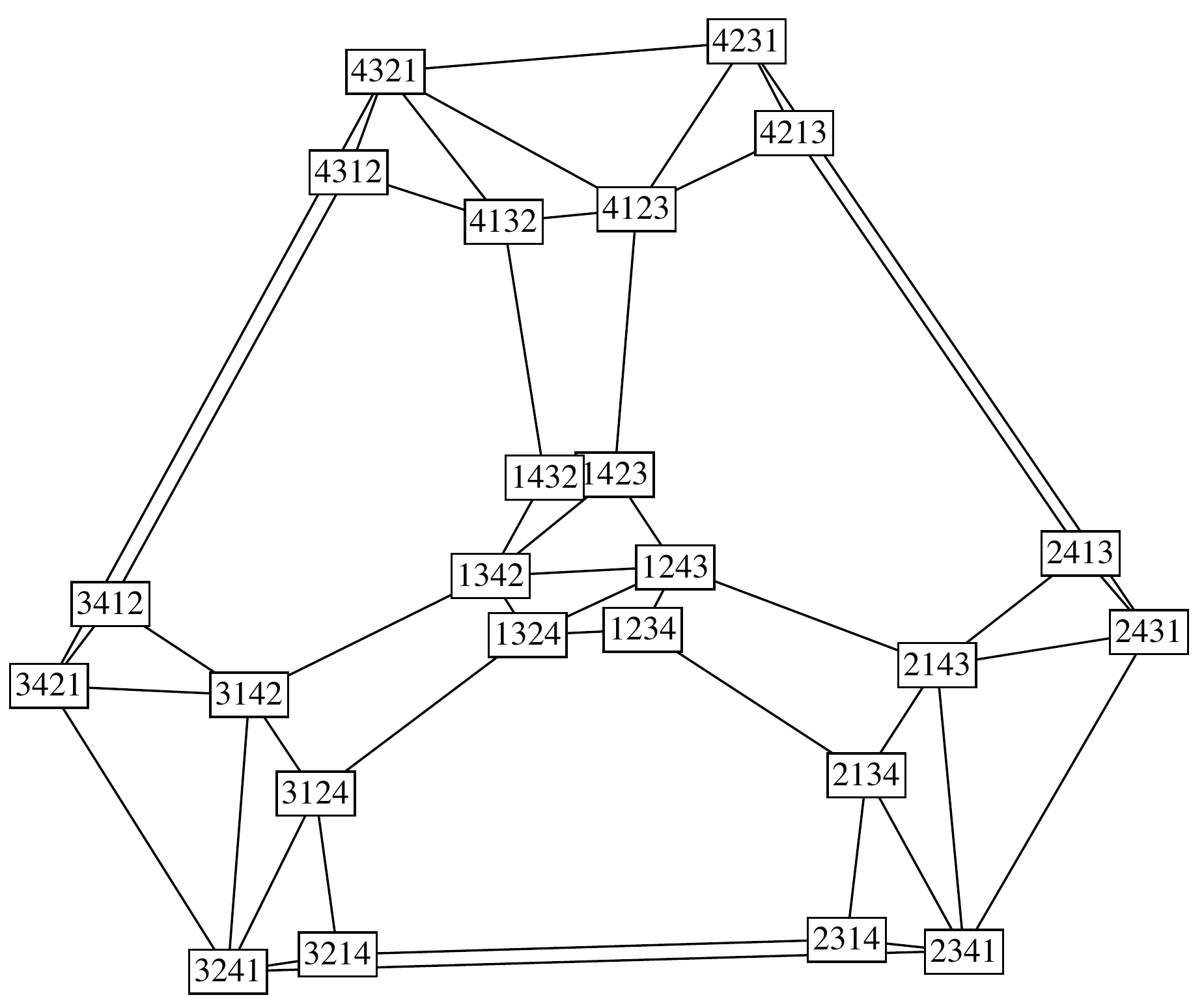}
\caption{Schlegel Diagram for the Affine Representation of $\mc{M}(PL_4)$}
\label{pl4fig}
\end{figure}

\begin{comment}
\begin{figure}[!h]
\includegraphics[scale=.5]{plackettLuce.png}
\caption{Schlegel Diagram for the Affine Representation of $\mc{M}(PL_4)$}
\label{pl4fig}
\end{figure}
\end{comment}

At first glance, this arrangement looks like the vertices of a permutahedron; however, some sets expected to be contained in facets are in fact full-dimensional. The polytope has four hexagonal facets given by fixing the last element of the ranking and acting with $S_3$ on the others. The four facets of the permutahedron corresponding to fixing the first element are triangulated. This may be due to the fact that in the parametrization the last element of the ranking does not make an appearance. (e.g. $p_{1234} \mapsto x_1(x_1 + x_2)(x_1 + x_2 + x_3)$).
The full matroid is too large for computation; instead, we use combinatorial tools in {\tt Sage} to find representatives of each base and circuit modulo the natural $S_4$-action on the set of variables. We will refer to distinct bases and circuits modulo the group action as \emph{base classes} and \emph{circuit classes}, respectively.

\begin{enumerate}[1.]
\item \emph{Decorated Circuits}: There are five circuit classes of size $4$ with orbit size $6,12,12,12,24$:

\begin{footnotesize}
\begin{longtable}{|p{.23\textwidth}|m{.5\textwidth}|c|} \hline
{\bf Type:} & {\bf Polynomial:} & {\bf Orbit Size:} \\ \hline
\hspace{5mm} \raisebox{-.5\height}{\includegraphics[scale = .55]{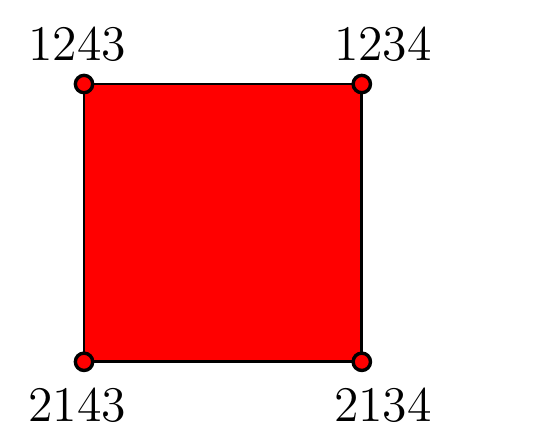}} & $p_{1243}p_{2134}-p_{1234}p_{2143}$ & 6  \\ \hline
\raisebox{-.5\height}{\includegraphics[scale = .5]{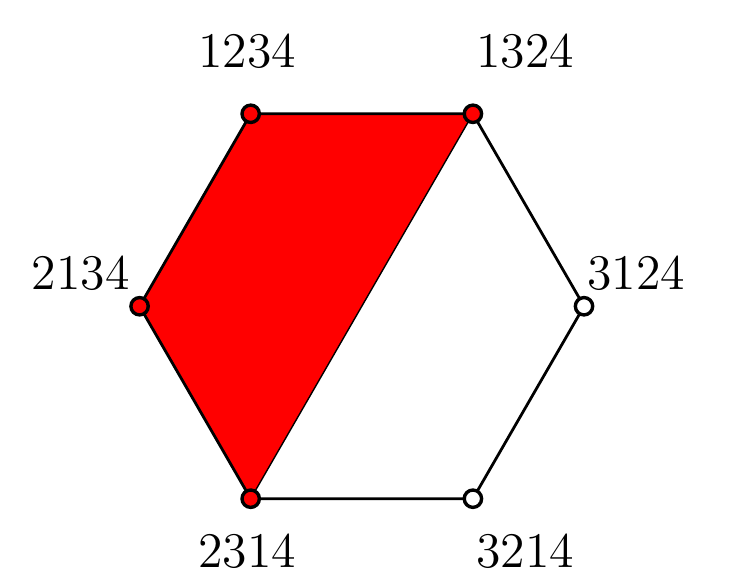}} & $p_{1234}^2p_{2134}-p_{1234}p_{1324}p_{2134}-p_{1234}p_{2134}^2-p_{1324}p_{2134}^2+p_{1234}^2p_{2314}+p_{1234}p_{2134}p_{2314}$ & 12 \\ \hline

\raisebox{-.5\height}{\includegraphics[scale = .5]{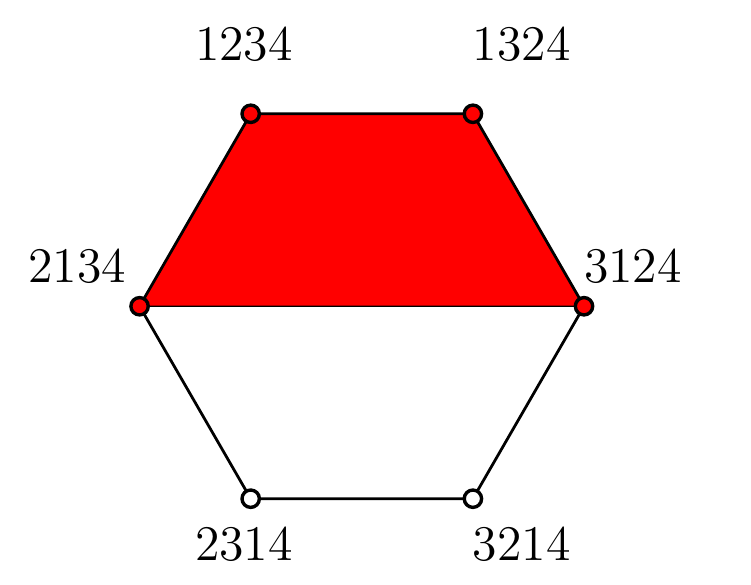}} & $p_{1234}^2p_{1324}-p_{1234}p_{1324}^2-p_{1324}^2p_{2134}+p_{1234}^2p_{3124}$ & 12 \\ \hline

\raisebox{-.5\height}{\includegraphics[scale = .5]{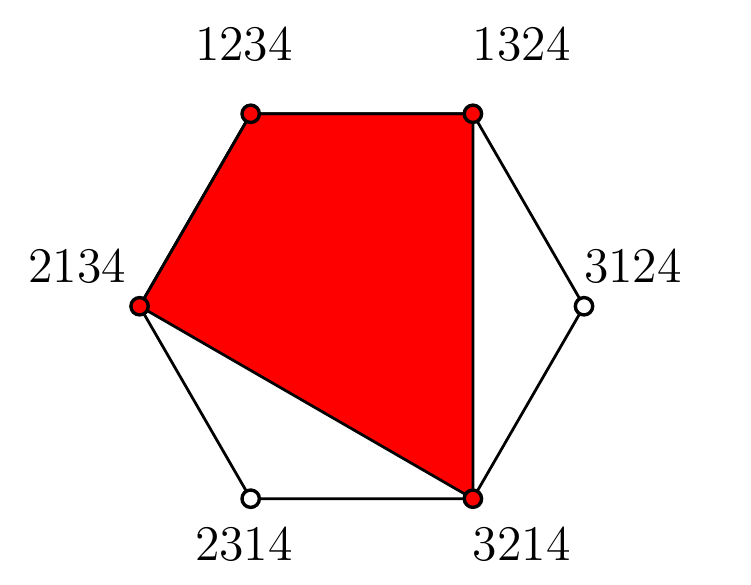}} &
 $p_{1234}^4-2p_{1234}^3p_{1324}+p_{1234}^2p_{1324}^2-p_{1234}^3p_{2134}-p_{1234}^2p_{1324}p_{2134}+2p_{1234}p_{1324}^2p_{2134}+p_{1234}p_{1324}p_{2134}^2+p_{1324}^2p_{2134}^2-p_{1234}^3p_{3214}-p_{1234}^2p_{2134}p_{3214}$ & 24 \\ \hline
 
 \raisebox{-.5\height}{\includegraphics[scale = .5]{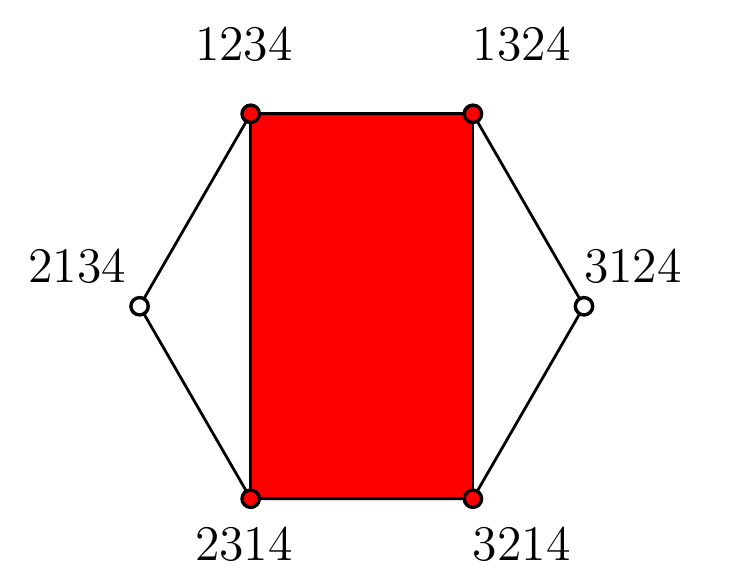}} &
 $p_{1234}^2p_{2314}-2p_{1234}p_{1324}p_{2314}+p_{1324}^2p_{2314}-p_{1324} p_{2314}^2 + p_{1234}^2p_{3214}- 2p_{1234}p_{1324}p_{3214} +p_{1324}^2p_{3214}+ p_{1234}p_{2314}p_{3214}+ p_{1324}p_{2314}p_{3214}-p_{1234}p_{3214}^2$ & 12 \\ \hline
\end{longtable}
\end{footnotesize}

There are 1720 circuit classes of size $5$, each of which has orbit size $24$, yielding an additional $41,280$ circuits. Important to note here: the Macaulay2 computation ran for 10 days without producing circuit polynomials. Bertini was able to produce a witness set in approximately 8 hours and compute 1720 projections in approximately 6 hours (working in parallel). The degrees of the circuit polynomials are recorded in Figure \ref{fig:pl4Circs}.

\item \emph{Decorated Bases}: The 464 classes of size $4$ that are not circuits are bases. The computation of base decorations produces the distribution of base degrees in Figure \ref{Fig:pl4bases}.

\begin{figure}[!h]
\begin{minipage}{.48\textwidth}
\centering
\begin{tikzpicture}
\begin{axis}[ybar, x tick label style={/pgf/number format/1000 sep=}, bar width=6pt]
\addplot plot coordinates
{(3,4) (4,14) (5,8) (6,64) (7,69) (8,92) (9,119) (10,161) (11,151) (12,140) (13,111) (14,157)
     (15,119) (16,96) (17,89) (18,97) (19,62) (20,52) (21,45) (22,17) (23,21) (24,20) (25,9) (26,3)};
\end{axis}
\end{tikzpicture}
\caption{Circuit degree frequency.}
\label{fig:pl4Circs}
\end{minipage} \begin{minipage}{.5\textwidth}
\centering
\begin{tikzpicture}
\begin{axis}[ybar, x tick label style={/pgf/number format/1000 sep=}, bar width=7pt]
\addplot plot coordinates
{(1,18) (2,43) (3,61) (4,62) (5,37) (6,73) (7,33) (8,29) (9,28) (10,16) (11,18) (12,20)
       (13,2) (14,2) (15,4) (16,7) (17,4) (18,4) (19,0) (20,1) (21,0) (22,0) (23,1) (24,1)};
\end{axis}
\end{tikzpicture}
\caption{Base degree frequency.}
\label{Fig:pl4bases}
\end{minipage}
\end{figure}

\begin{comment}
\begin{tikzpicture}
\begin{axis}[ybar, x tick label style={/pgf/number format/1000 sep=}, bar width=7pt]
\addplot plot coordinates
{(0,324) (1,924) (2,1356) (3,1404) (4,888) (5,1722) (6,756) (7,678) (8,672) (9,354) (10,408)
        (11,480) (12,48) (13,48) (14,96) (15,156) (16,96) (17,96) (18,0) (19,24) (20,0) (21,0)
        (22,24) (23,6)};
\end{axis}
\end{tikzpicture}
\end{comment}

The highest base degree is $24$, which is also the cardinality of the matroid, and the size of the symmetry group. The base of degree 24 is $\{p_{1234}, p_{2341}, p_{3412}, p_{4123}\}$. In other words, pick a ranking and apply the 4-cycle to it.  The degree of the variety, which tells us the degree of a fiber under generic projection, is $27$, indicating that all of the coordinate projections are ``special."
\end{enumerate}
Knowing about the decorated bases and circuits of this matroid allows us to understand its coordinate projections, and gives valuable information about reconstructing partial data.

\end{ex}

Another application of matroids to algebraic statistics is in the study of mixture models. The $r$-th mixture model of a pair of discrete random variables $X$ and $Y$, with $m$ and $n$ states respectively, models the situation where $X\independent Y$ conditional on a ``hidden" variable $Z$ which occupies $r$ states. In \cite{KRS14}, the algebraic boundary of the mixture model for $m = n = 4$ and $r =3$ is computed; it has 288 components, one of which is analyzed in the example. Studying this example gives insight into the combinatorics of \emph{all} of the components of the variety, in addition to the independence structure of this particular component.

\begin{ex}[Mixture Model Matroid] \label{kaie} $\begin{array}{|ccccc|} \hline
\mc{M}(I_{mix}): & |E| = 16, & \rho = 14, & |\mc{B}| = 112, & |\mc{C}| = 11. \\ \hline
\end{array}$
We examine one of the components of the algebraic boundary of the mixture model of rank $3$ for $4 \times 4$ matrices, as defined in \cite[Example 5.2]{KRS14}. Let $I_{mix}$ denote the defining ideal of this component; $I_{mix}$ is generated by the $4 \times 4$ minors of the following matrix:
\[ \left( \begin{array}{ccccc}
p_{11} & p_{12} & p_{13} & p_{14} & 0 \\
p_{21} & p_{22} & p_{23} & p_{24} & 0 \\
p_{31} & p_{32} & p_{33} & p_{34} & p_{33} (p_{11}p_{22} - p_{12}p_{21}) \\
p_{41} & p_{42} & p_{43} & p_{44} & p_{41} (p_{12} p_{23} - p_{13} p_{22}) + p_{43} (p_{11}p_{22} - p_{12}p_{21})
\end{array} \right)
\]
\begin{enumerate}[1.]
\item \emph{Decorated Bases}:
The base enumeration in this case is very quick. There are $120$ subsets of size $14$; checking all of them took $0.5$ seconds to run in Macaulay2 before returning a list of $112$ bases.

The cobases are the pairs of variables for which the corresponding edge {\bf is not} one of the 8 edges in Figure \ref{Fig:mix}.

\begin{figure}[h]
\begin{minipage}{.5\textwidth}
\centering
\includegraphics[scale=1]{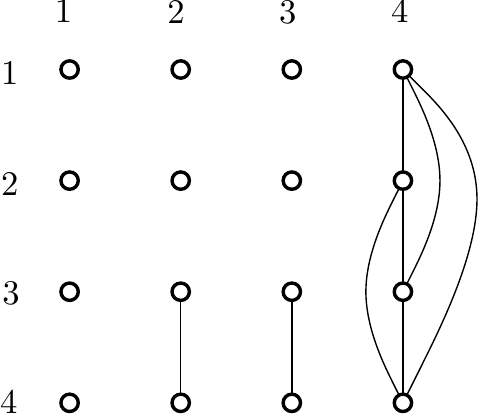}
\caption{Non-cobases of $\mc{M}(I_{mix})$.}
\label{Fig:mix}
\end{minipage} \begin{minipage}{.48\textwidth}
\centering
\includegraphics[scale=1]{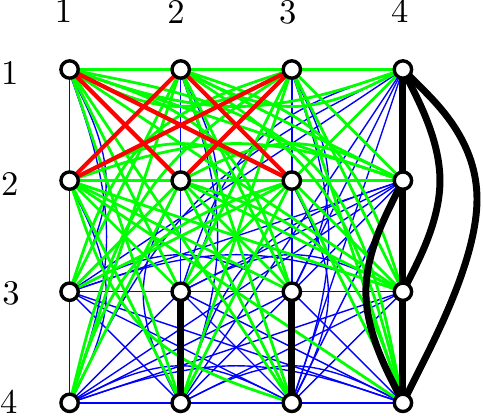}
\caption{Base degrees for $\mc{M}(I_{mix})$.}
\label{Fig:mix2}
\end{minipage}
\end{figure}

Computation of base degree yields the following numbers:
\[ \begin{array}{|l|ccc|} \hline
\text{\bf Base Degree} & 1 & 2 & 3 \\ \hline
\# \: \text{\bf  of Bases} & 52 & 54 & 6 \\ \hline
\end{array} \]
The blue edges in Figure \ref{Fig:mix2} indicate the complements of degree-$1$ bases, the green edges are the degree-$2$ bases, and the six red edges are the degree-$3$ bases.

\vspace{5mm}

\item \emph{Decorated Circuits}: The matroid has $11$ circuits, which can be read from Figure~\ref{Fig:mix} as the complements of each connected component of the graph; this coincidence is because the cohyperplanes have rank one. The computation checking all sets of size $\leq 11$ took 3090 seconds. When we started instead from the fundamental circuits associated to some base, and checked mutual eliminations (the Boros et al algorithm), took only 2.7 seconds to produce all circuits, and another 293 seconds to verify that the list was complete. This is one example where, due to high rank, the alternative circuit enumeration algorithm is preferred. The circuit polynomials are too big to record here: the number of terms in each polynomial are 24, 27, 27, 19, 150, 136, 24, 136, 150, 150, and 150, respectively. Instead, we record relevant statistics:

\begin{longtable}{|l|l|l|} \hline
{\bf Circuit Complement} & {\bf Circuit Top-Degree} & {\bf Circuit Degree} \\ \hline
$p_{32}, p_{42}$ &$(2,1,2,1,2,1,2,1,1,0,1,1,1,0,1,1)$ & $6$ \\
$p_{41}$ & $(1,2,2,1,1,2,2,1,1,1,1,1,0,1,1,1)$ &$6$\\
$p_{31}$ & $(1,2,2,1,1,2,2,1,0,1,1,1,1,1,1,1)$ & $6$\\
$p_{34}, p_{44}, p_{14},p_{24}$ & $(2,2,2,0,2,2,2,0,1,1,1,0,1,1,1,0)$ & $6$\\
$p_{22}$ & $(3,1,3,2,2,0,2,2,2,1,2,2,2,1,2,2)$ & $9$\\
$p_{21}$ & $(1,3,3,2,0,2,2,2,1,2,2,2,1,2,2,2)$ &$9$\\
$p_{33}, p_{43}$ &  $(2,2,1,1,2,2,1,1,1,1,0,1,1,1,0,1)$ & $6$\\
$p_{11}$ & $(0,2,2,2,1,3,3,2,1,2,2,2,1,2,2,2)$ & $9$\\
$p_{12}$ &   $(2,0,2,2,3,1,3,2,2,1,2,2,2,1,2,2)$ & $9$\\
$p_{13}$ & $(2,2,0,2,3,3,1,2,2,2,1,2,2,2,1,2)$ & $9$\\
$p_{23}$ &  $(3,3,1,2,2,2,0,2,2,2,1,2,2,2,1,2)$ & $9$ \\\hline
\end{longtable}

\vspace{-2mm}

These circuit statistics tell us how many completions are possible for every projection.
\end{enumerate}
The combinatorial characterization of this \emph{component} also carries information for the \emph{global} structure. Consider the action of $S_4 \times S_4 \times \ZZ_2$ on the labeled graph of Figure \ref{Fig:mix}. The orbit of the graph has cardinality 144. 
Indeed, in \cite[Example 5.2]{KRS14}, it is shown that the $288$ components are paired by taking the transpose of the factorization, and both components in a given pair have the same matroid.

\vfill
\pagebreak
\noindent 
\end{ex}

%\vfill

\subsection{Chemical Reaction Networks}
In the algebraic study of chemical reaction networks (CRNs), steady-state concentrations of chemical species lie in some algebraic variety. The matroid associated to this variety may be used to design an experiment where measurements of each coordinate are obtained only through some specified costs.  Then, bases of the matroid would be appropriate to measure if the goal is to find all concentrations; if we aim to test the validity of the model, a circuit may be a good choice for model rejection. (For more details, see the author's upcoming article ``Matroids for Experimental Design" with Harrington.)

\begin{ex}[MAPK Network]  $\begin{array}{|ccccc|} \hline
\mc{M}(I_{MAPK}): & |E| = 12, & \rho = 3, & |\mc{B}| = 190, & |\mc{C}| = 303. \\ \hline
\end{array}$

This ideal comes from \cite{MG08}, which analyzes the polynomials defining the steady-state of a certain CRN. Each variable corresponds to the concentration of some chemical species:

\[ R = \RR[KS_{00}, KS_{01}, KS_{10}, FS_{01}, FS_{10}, FS_{11}, K, F, S_{00}, S_{01}, S_{10}, S_{11}].\]

 \begin{scriptsize}
 $I_{MAPK} =  \langle a_{00}\cdot K\cdot S_{00}+b_{00}\cdot KS_{00}+\gamma_{0100}\cdot FS_{01}  +\gamma_{1000}\cdot FS_{10}+\gamma_{1100}\cdot FS_{11},  
 -a_{01}\cdot K\cdot S_{01}+b_{01}\cdot KS_{01}+c_{0001}\cdot KS_{00}  -\alpha_{01}\cdot F\cdot S_{01}+\beta_{01}\cdot FS_{01}+\gamma_{1101}\cdot FS_{11},  
 -a_{10}\cdot K\cdot S_{10}+b_{10}\cdot KS_{10}+c_{0010}\cdot KS_{00}  -\alpha_{10}\cdot F\cdot S_{10}+\beta_{10}\cdot FS_{10}+\gamma_{1110}\cdot FS_{11},  
 -\alpha_{11}\cdot F\cdot S_{11}+\beta_{11}\cdot FS_{11}  +c_{0111}\cdot KS_{01}+c_{1011}\cdot KS_{10}+c_{0011}\cdot KS_{00},  
 a_{00}\cdot K\cdot S_{00}-(b_{00}+c_{0001}+c_{0010}+c_{0011})\cdot KS_{00},   
a_{01}\cdot K\cdot S_{01}-(b_{01}+c_{0111})\cdot KS_{01}, 
 a_{10}\cdot K\cdot S_{10}-(b_{10}+c_{1011})\cdot KS_{10},  
\alpha_{01}\cdot F\cdot S_{01}-(\beta_{01}+\gamma_{0100})\cdot FS_{01},  
 \alpha_{10}\cdot F\cdot S_{10}-(\beta_{10}+\gamma_{1000})\cdot FS_{10}, 
 \alpha_{11} \cdot F\cdot  S_{11}-(\beta_{11}+\gamma_{1101}+\gamma_{1110}+\gamma_{1100})\cdot  FS_{11},  
 -a_{00}\cdot  K\cdot  S_{00}+(b_{00}+c_{0001}+c_{0010}+c_{0011})\cdot KS_{00}
 -a_{01}\cdot K\cdot S_{01} + (b_{01}+c_{0111})\cdot KS_{01}-a_{10}\cdot K\cdot S_{10}  +(b_{10}+c_{1011})\cdot KS_{10},  
 -\alpha_{01}\cdot F\cdot S_{01}+(\beta_{01}+\gamma_{0100})\cdot FS_{01}-\alpha_{10}\cdot F\cdot S_{10}+  (\beta_{10} +\gamma_{1000})\cdot FS_{10}-\alpha_{11}\cdot F\cdot S_{11} +  (\beta_{11}+\gamma_{1101}+\gamma_{1110}+\gamma_{1100})\cdot FS_{11} \rangle.$
\end{scriptsize} 

\vspace{5mm}

The $a,b, c, \alpha,\beta$, and $\gamma$ constants are taken to be random real numbers, i.e. a set of algebraically independent transcendentals over $\QQ$. If the rate constants are originally taken to be part of the matroid, this specialization amounts to matroid contraction. The ideal $I_{MAPK}$ is radical with two associated primes: (1) a variety of degree 10 and dimension 3, and 
(2) a coordinate subspace with ideal $\langle F, K, FS_{11}, FS_{10}, FS_{01}, KS_{1011}, KS_{01}, KS_{00} \rangle$. In the chemical reaction, the latter component to the steady-state achieved by the disappearance of these reactants.
We are interested in the rank $3$ matroid associated to the former component. A quick symbolic calculation determines that the matroid has affine representation as in Figure \ref{Fig:mapk}.

\begin{figure}[!h]
\includegraphics[scale=0.8]{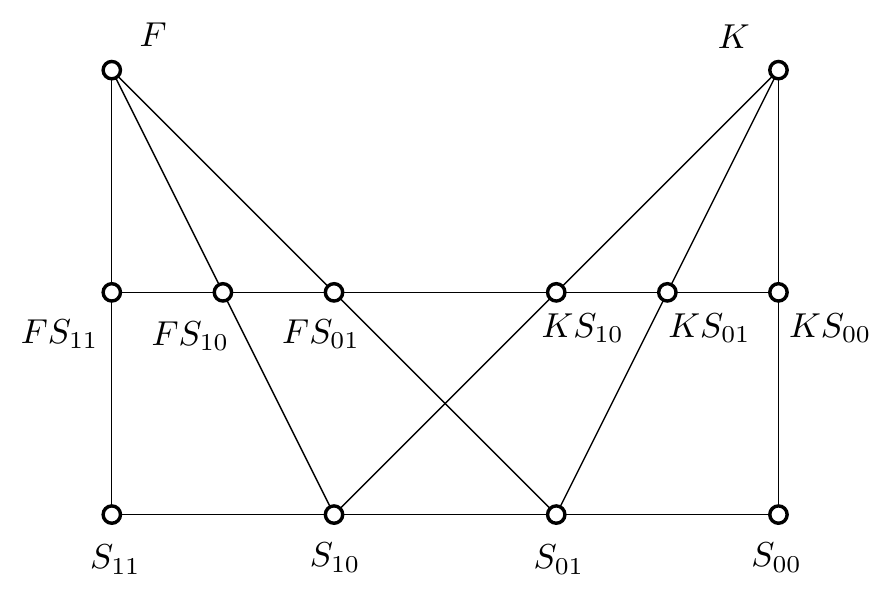}
\caption{Affine representation of the MAP Kinase matroid.}
\label{Fig:mapk}
\end{figure}

\begin{enumerate}[1.]
\item \emph{Decorated Bases}: Any non-collinear set of $3$ elements from the diagram are a basis of the matroid; there are 190 in total. Of these, $52$ have base degree $1$, $124$ have degree $2$, and $14$ have degree $3$.

\item \emph{Decorated Circuits}: There are circuits of size $3$ and $4$. The size $3$ circuits are the $30$ collinear sets of $3$: these have degree $2$ except for $\{S_{00},S_{01},S_{11}\}$ and $\{S_{00},S_{10},S_{11}\}$, which have degree $3$.

There are $273$ circuits of size $4$; the degrees of the circuit polynomials have the following frequencies:

\[ \begin{array}{|l|ccccc|} \hline
\text{\bf Circuit Degree} & 2 & 3 & 4 & 5 & 6 \\ \hline
\# \: \text{\bf  of Circuits} & 13 & 76 & 125 & 49 & 10 \\ \hline
\end{array} \]

\end{enumerate}
Possessing this data aids in experimental design, as mentioned above; however, it also distills the combinatorial essence of a chemical reaction network. This demonstrates the power of algebraic matroids in summarizing structure.
\end{ex}

\subsection{The Grassmannian} 

In algebraic geometry and representation theory, some important objects have a distinguished set of coordinates. For the Grassmannian $Gr(r,n)$, the Pl\"{u}cker coordinates are the variables of choice. When $r  = 2$, the Grassmannian is defined by skew-symmetric $n\times n$ matrices; this is thematically similar to the content of \cite{KRT13}. So, we examine the next case of interest: $r =3$. 

\begin{ex} $\begin{array}{|ccccc|} \hline
\mc{M}(Gr(3,6)): & |E| = 20, & \rho = 10, & |\mc{B}| = 184,590, & |\mc{C}| = 51,005. \\ \hline
\end{array}$

$Gr(3,6)$ is the variety of $3$-dimensional subspaces of $\CC^6$, with coordinates given by the Pl\"{u}cker coordinates $p_{ijk}$, with $1 \leq i,j,k \leq 6$, distinct.
The ideal of the Grassmannian is generated by $35$ Pl\"{u}cker relations of degree $2$.
\begin{enumerate}[1.]
\item \emph{Decorated Bases}: The bases are sets of size $10$. Computation is aided here by using {\tt Sage} to give only one representative of each class up to the $S_6$ action on the Grassmannian. The rest is carried out in $7$ seconds by {\tt Macaulay2}. There are $197$ base classes of degree $1$, $42$ of degree $2$, two of degree $3$, and one of degree $7$.

%\[\begin{array}{|r|ccccccc|} \hline \text{\bf Base Degree} & 1&  2 & 3 &4 &5 &6 &7 \\ \hline
%\text{\bf \# of Bases} &197 & 42 & 2 &0 & 0 & 0 & 1\\ \hline
%\end{array}\]

The degree-$7$ base appears to be an outlier, so further examination seems appropriate. The appearing variables correspond to the triangles in the beautifully symmetric complex in Figure \ref{Fig:grBase}. This image is familiar as a minimal triangulation of $\RR \PP^2$; we plan to explore the connection between this image and high-degree projections of the Grassmannian in future work.

%\vspace{-2mm}
\begin{figure}[!h]
\centering
\includegraphics[scale=.52]{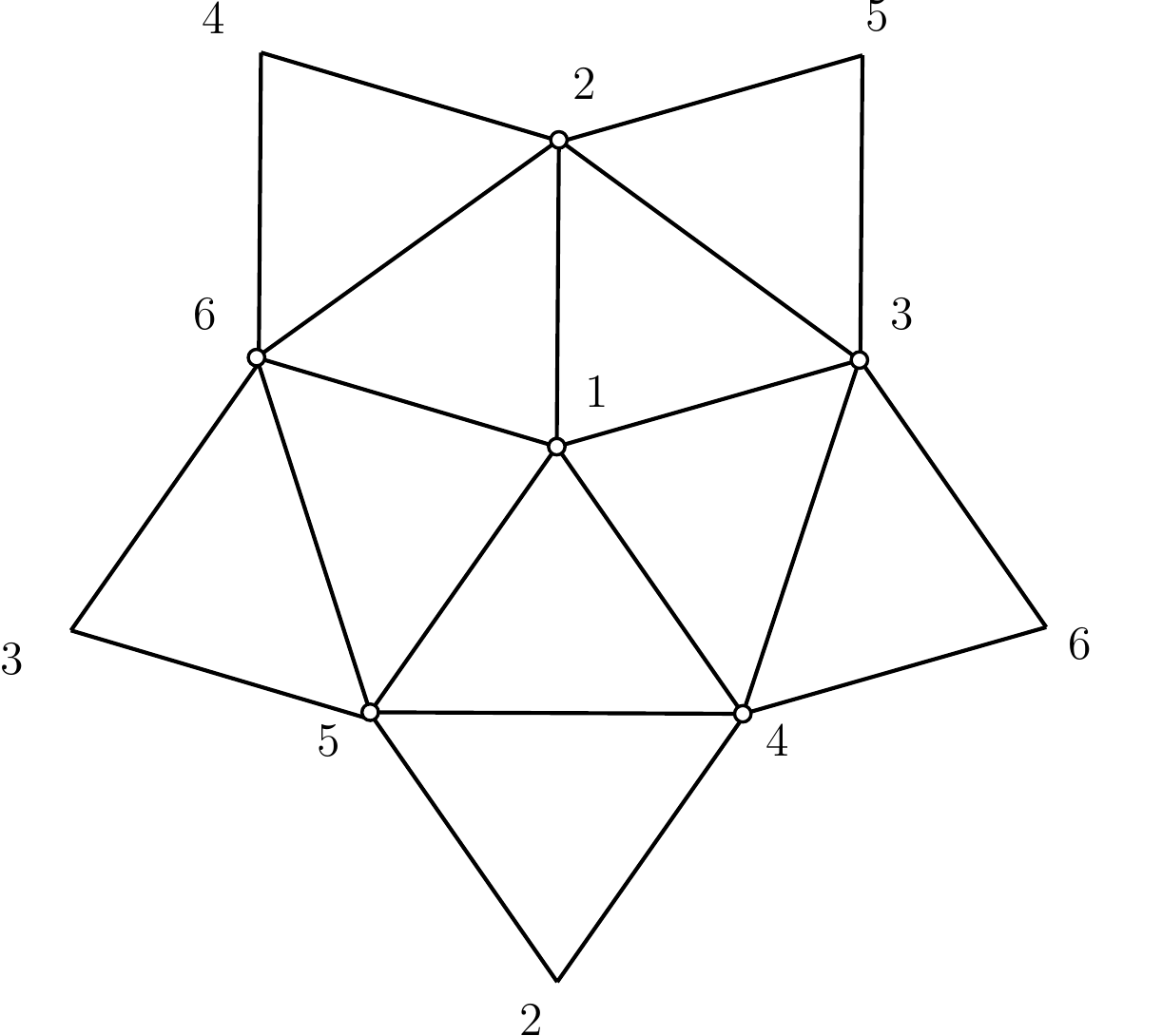}
\caption{Base of $\mc{M}(Gr(3,6))$ with base degree $7$.}
\label{Fig:grBase}
\end{figure}

\item \emph{Decorated Circuits}: For circuit computation, {\tt Sage} once again proved vital in cutting down the number of required tests by a factor of approximately $6!$. The testing on the circuit class representatives took $55$ seconds in Macaulay2 before returning the list of circuits.  There are $97$ total circuit classes, with degree of circuit polynomials distributed as in Figure \ref{grCircs}. Taking into account the full orbits of each circuit, we have $51,005$ total circuits. The only circuit class of degree $12$ is obtained from our special base by adding one triangle, e.g. the variable $p_{456}$ as in Figure \ref{Fig:grCirc}.

%\vspace{-2mm}

\begin{figure}[!h]
\begin{minipage}{.48\textwidth}
\centering
\begin{tikzpicture}
\begin{axis}[height = 6cm, ybar, x tick label style={/pgf/number format/1000 sep=}, bar width=11pt]
\addplot plot coordinates
{(2, 2) (3, 8) (4, 22) (5, 40) (6, 14) (7, 9) (8, 1) (9, 0) (10, 0) (11, 0) (12, 1)};
\end{axis}
\end{tikzpicture}
\caption{Circuit degree frequency.}
\label{grCircs}
\end{minipage} \begin{minipage}{.5\textwidth}
\centering
\includegraphics[scale=.5]{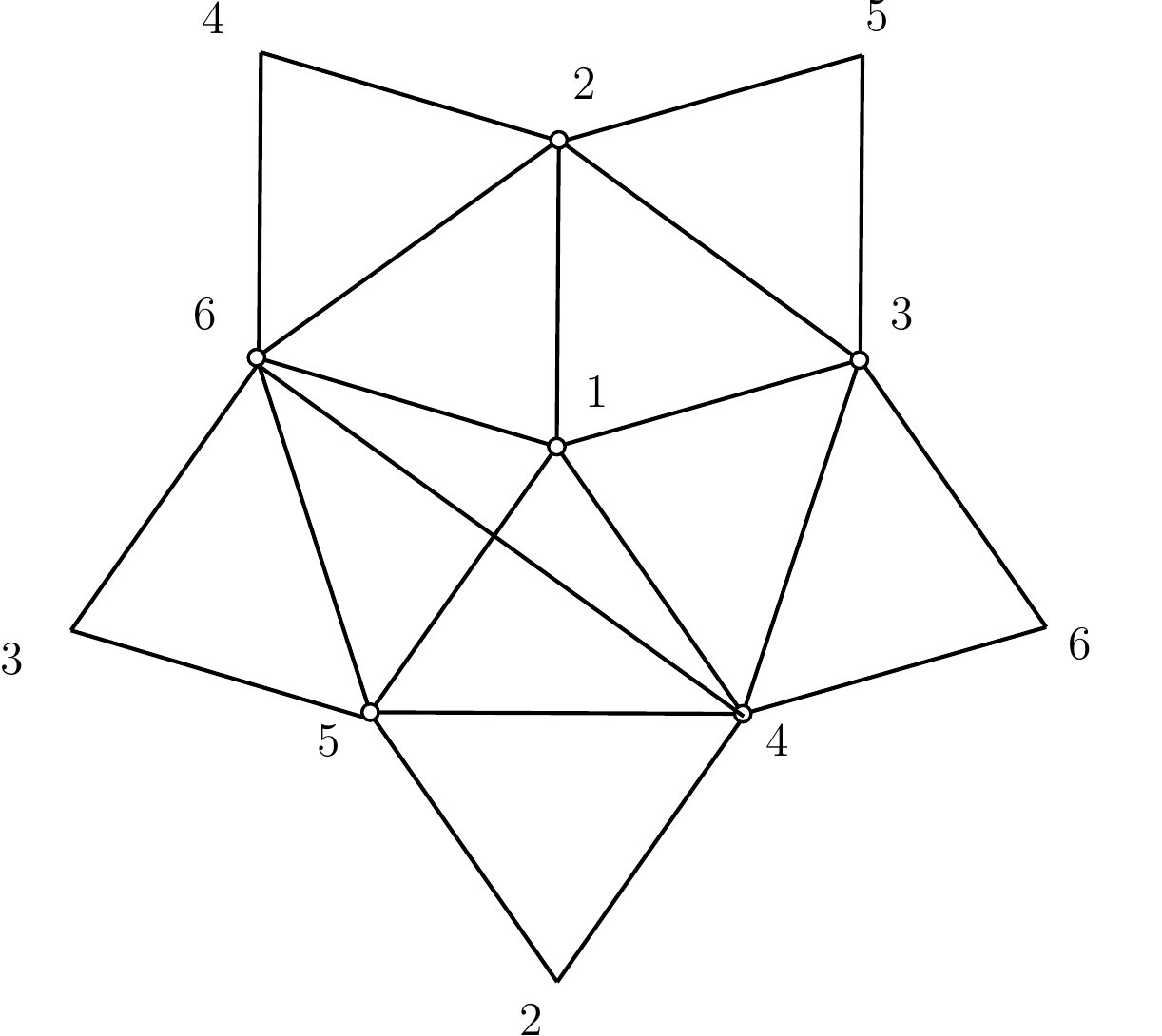}
\caption{Circuit of $\mc{M}(Gr(3,6))$ with degree $12$.}
\label{Fig:grCirc}
\end{minipage}
\end{figure}
\end{enumerate}

\end{ex}

%\vspace{-6mm}

\subsection{Matroid Representations} \label{sec:matrep} There is a small collection of algebraic matroids over finite fields that are not representable as linear matroids.  Note that base degree is not well-defined for these algebraic matroids, since a ``generic fiber" is not well-defined. Still, computation of the corresponding ideal with circuit polynomials can give insight into the structure of the matroid.  One such matroid is explored in the example:

\begin{ex}[Non-Pappus Matroid]  $\begin{array}{|ccccc|} \hline
\mc{M}(I): & |E| = 9, & \rho = 3, & |\mc{B}| = 76, & |\mc{C}| = 86. \\ \hline
\end{array}$

\noindent The \emph{non-Pappus Matroid} is algebraic over every finite field while not being linearly representable over any field. Since linear representability $\iff$ algebraic representability for fields of char. $0$, this is as extreme as a matroid can be.

\begin{figure}[!h] \label{Fig:nonPappus}
\includegraphics[scale=.7]{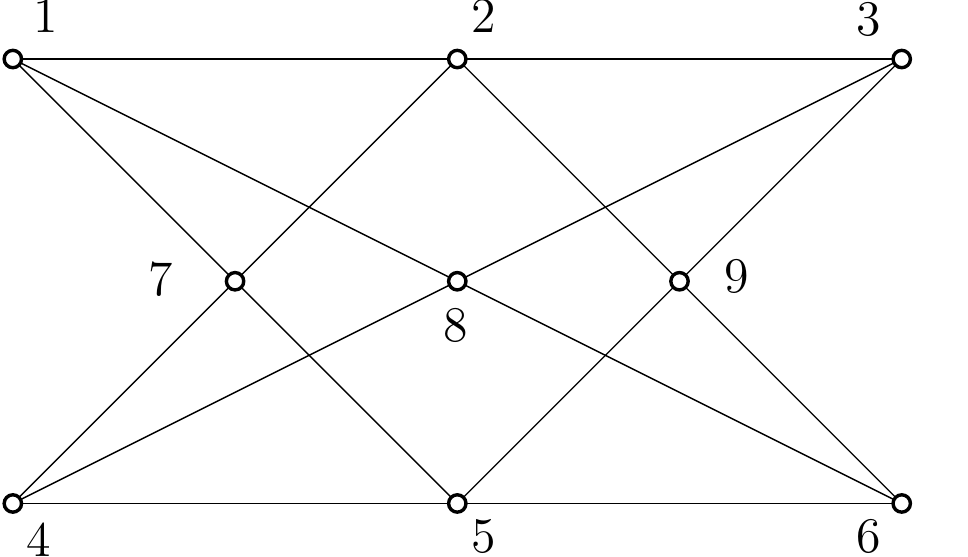}
\caption{Non-Pappus Matroid.}
\end{figure}

\centerline{
\begin{minipage}{0.4\textwidth}
Algebraic Matroid over $\FF_4 \: (\lambda \neq \lambda^2)$:
\[
\begin{array}{rcl}
\varphi(1) & = & x^2 + y, \\
\varphi(2) & = &	  x, \\
\varphi(3) & = &	  x + y,\\ 
\varphi(4) & = &	  y + z, \\
\varphi(5) & = &	  y +  \lambda z, \\
\varphi(6) & = &	  z, \\
\varphi(7) & = &	  ( \lambda-1) x^2 +  \lambda   y +  \lambda   z,\\
\varphi(8) & = &	  x^2 + y + z - z^2,\\
\varphi(9) & = &   \lambda z - x.\\
\end{array}
\]
\end{minipage} \hspace{5mm}
\begin{minipage}{0.4\textwidth}
Algebraic Matroid over $\FF_2$:
\[
\begin{array}{rcl}
\varphi(1) & = & x, \\
\varphi(2) & = &	  x+y, \\
\varphi(3) & = &	 y, \\ 
\varphi(4) & = &	  x + y + \frac{xz}{x+y}, \\
\varphi(5) & = &	  z, \\
\varphi(6) & = &	    x + y + \frac{yz}{x + y} , \\
\varphi(7) & = &	  xz,\\
\varphi(8) & = &	  xy +  \frac{xyz}{x+y} ,\\
\varphi(9) & = &  yz.\\
\end{array}
\]
\end{minipage}}

\vspace{3mm}

The algebraic representation on the right was used by Lindstrom in \cite{Lind83} to prove that the non-Pappus matroid is algebraic. The algebraic representation on the left is a valid algebraic representation over $\FF_{p^2}$ for any $p$ prime, used in \cite{Lind86} to prove an infinite algebraic characteristic set. We can be a bit more precise in assessing these matroid representations, by computing the implicit ideal of each.

The representation over $\FF_2$ has defining ideal generated as:
\[\langle t_4+t_5+t_6, t_1+t_2+t_3, t_5t_8+t_3t_9+t_5t_9+t_6t_9, t_3t_7+t_2t_9+t_3t_9, \] \vspace{-6mm} \[ t_2t_7+t_6t_7+t_2t_9+t_5t_9+t_6t_9, t_3t_5+t_9, t_2t_5+t_7+t_9, t_3^2+t_3t_6+t_8+t_9, t_2^2+t_2t_6+t_9 \rangle. \]
Compiling the degrees of the circuit polynomials, we have:
\[ \begin{array}{|l|ccccccc|} \hline
\text{\bf Degree} & 1 & 2 & 3 & 4 & 5 & 6 & 7 \\ \hline
\# \: \text{\bf  of Circuits} & 2 & 33 & 24 & 21 & 4 & 0 & 2 \\ \hline
\end{array} \]

The representation over $\FF_4$ has defining ideal generated as:
\[ \langle t_4+t_5+(\lambda+1)t_6, t_3+t_5+t_9, t_2+\lambda t_6+t_9, t_1+\lambda t_5+\lambda t_7,\] \vspace{-6mm}  \[ t_9^2+\lambda t_5+t_6+(\lambda+1)t_7+(\lambda+1)t_8, t_6^2+\lambda t_5+t_6+\lambda t_7+t_8 \rangle .\]
The degrees of the circuit polynomials appear with the following frequency:
\[ \begin{array}{|l|cccc|} \hline
\text{\bf Degree} & 1 & 2 & 3 & 4  \\ \hline 
\# \: \text{\bf  of Circuits} & 12 & 59 & 0 & 15 \\ \hline
 \end{array} \]

Further examination of the {\it decorated} algebraic matroid may give insight into the various possible representations of this and similar nonlinear matroids.

\end{ex}

{\bf Acknowledgments}: Thanks to Bernd Sturmfels for suggesting the project and providing most of the examples. Thanks to Franz Kir\'{a}ly, Louis Theran, and Alex Fink for helpful conversations.  Thanks to Dan Bates, Jon Hauenstein, and Jose Rodriguez for their help with Bertini, and to Ariel Allon and Shivaram Lingamneni for their expert coding advice. This research was supported by the NSF through grant DMS-0968882, and by the Max Planck Institute for Mathematics in Bonn.

\bibliography{compbib}{}
\bibliographystyle{plain}

\end{document}